\documentclass[12pt]{amsart}
\usepackage{amssymb}
\usepackage{amsmath,amssymb,amsfonts,amsthm,latexsym}
\usepackage{mathrsfs}
\usepackage{color}
\usepackage{hyperref}
\usepackage{graphicx}

\theoremstyle{plain}
\newtheorem{theorem}{Theorem}

\newtheorem{lemma}{Lemma}
\newtheorem{proposition}{Proposition}

\theoremstyle{definition}

\theoremstyle{example}

\theoremstyle{remark}

\numberwithin{equation}{section}

\begin{document}

\begin{center}
{\bf\Large The topological filtration of $\gamma$-structures}
\\
\vspace{15pt} Thomas J. X. Li, Christian M. Reidys$^{\,\star}$
\end{center}

\begin{center}
Institut for Matematik og Datalogi,\\
University of Southern Denmark,\\
 Campusvej 55, DK-5230 Odense M, Denmark\\
E-mail: duck@santafe.edu
\end{center}
\centerline{\bf Abstract}{\small
In this paper we study $\gamma$-structures filtered by topological genus.
$\gamma$-structures are a class of RNA pseudoknot structures that plays
a key role in the context of polynomial time folding of RNA pseudoknot
structures. A $\gamma$-structure is composed by specific building
blocks, that have topological genus less than or equal to $\gamma$,
where composition means concatenation and nesting of such blocks.
Our main results are the derivation of a new bivariate generating function
for $\gamma$-structures via symbolic methods, the singularity analysis of
the solutions and a central limit theorem for the distribution of
topological genus in $\gamma$-structures of given length.
In our derivation specific bivariate polynomials play a central role.
Their coefficients count particular motifs of fixed topological genus
and they are of relevance in the context of genus recursion and novel
folding algorithms.}




\section{Introduction}

An RNA sequence is a linear, oriented sequence of the nucleotides (bases)
{\bf A,U,G,C}. These sequences ``fold'' by establishing bonds between
pairs of nucleotides. These bonds cannot form arbitrarily, a nucleotide can
at most establish one bond and the global conformation of an RNA molecule
is determined by topological constraints encoded at the level of secondary
structure, i.e., by the mutual arrangements of the base pairs \cite{Bailor:10}.
Secondary structures can be interpreted as (partial) matchings in a graph of
permissible base pairs \cite{Tabaska:98}. When represented as a diagram,
i.e.~as a graph whose vertices are drawn on a horizontal line with arcs in
the upper halfplane on refers to a secondary structure with crossing arcs
as a pseudoknot structure.

Folded configurations are energetically optimal. Here energy means free
energy, which is dominated by the stacking of adjacent base pairs and
not by the hydrogen bonds of the individual base pairs \cite{Mathews:99},
as well as minimum arc-length conditions \cite{Waterman:78aa}.
That is, a stack is tantamount to a sequence of parallel arcs
$((i,j),(i+1,j-1),\dots,(i+\tau-1,j-\tau+1))$.
In particular, only configurations without isolated bonds and without bonds
of length one (formed by immediately subsequent nucleotides) are observed in
RNA structures. For a given RNA sequence polynomial-time dynamic programming
(DP) algorithms can be devised, finding such minimal energy configurations.

The topological classification of RNA structures \cite{Bon:08,rnag3} has
recently been translated into an efficient dynamic programming algorithm
\cite{gfold}.
This algorithm {\it a priori} folds into a novel class of pseudoknot structures,
the $\gamma$-structures.
$\gamma$-structures differ from pseudoknotted RNA structures of fixed topological
genus of an associated fatgraph or double line graph \cite{Orland:02} and
\cite{Bon:08}, since they have not a fixed genus.
They are composed by irreducible subdiagrams whose individual genus is bounded by
$\gamma$ and contain no bonds of length one ($1$-arcs), see Section~\ref{S:Background}
for details. The folding of $\gamma$-structures has led to unprecidented sensitivity
and positive predictive value \cite{gfold}.

In this paper we study $\gamma$-structures filtered by topological genus,
i.e.~partial matchings composed by irreducible motifs of genus $\le \gamma$,
without $1$-arcs.
These motifs are called irreducible shadows and discussed in detail in
Section~\ref{S:Background}.
To consider the topological filtration of $\gamma$-structures is tantamount to
constructing a new bivariate generating function. This construction recruits
specific, bivariate generating polynomials that are associated to irreducible
shadows of genus $\leq \gamma$, which we denote by ${\bf Is}_{\gamma}(z,t)$.
For example for
$\gamma=1,2$ we have
\begin{eqnarray*}
{\bf Is}_{1}(z,t)&=&\left( 1+z \right) ^{2}{z}^{2}t,\\
{\bf Is}_{2}(z,t)&=& \left( 1+z \right) ^{4}{z}^{4}
       \left( 24\,z+17 \right)  \left( 4\,z+1
 \right) {t}^{2}+ \left( 1+z \right) ^{2}{z}^{2}t.
\end{eqnarray*}
The bivariate algebraic equations for $\gamma$-structures discovered here
are instrumental for obtaining recursions for computing shadows of genus
$g$ from those of smaller genera. Similar to the Zagier-Harer generating
function \cite{Zagier:95} it is a fascinating prospect to derive a recursion
for the polynomials ${\bf Is}_{g}(z,t)$. Here it will be vital to obtain hints
for bijective proof hidden in the algebraic formulas. Common factors of
these polynomials whose coefficients count numbers of irreducible shadows
of fixed genus will be the key for deeper understanding.
Results along these lines will have profound algorithmic impact and offer
novel insights in how to fold topological $\gamma$-structures faster.

We then study topological $\gamma$-structures from a probabilistic point
of view. Regarding the bivariate generating functions as parameterized
univariate functions we can prove a central limit theorem for the
distribution of topological genera in $\gamma$ structures of fixed
length $n$. We find that the expected genus of a canonical $1$-structure,
i.e.~a structure that does not contain any isolated arcs, is given by
$0.04123\ n$ with a variance of $0.0093 \ n$. Thus the expected genus is
linear in $n$ and in particular a canonical $1$-structure
of length $100$ has an expected genus of $4$, see Fig.~\ref{F:2}.

\begin{figure}
\begin{center}
\includegraphics[width=0.4\textwidth]{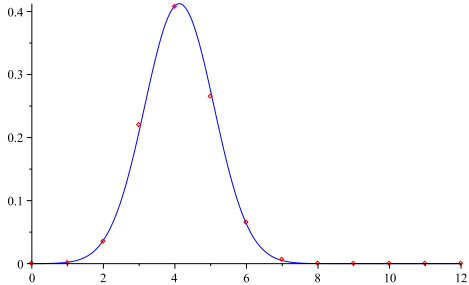}
\end{center}
\caption{\small The central limit distribution: we display the distribution of
topological genera of canonical $1$-structures for $n=100$.
}\label{F:2}
\end{figure}


\section{Background}


\subsection{$\gamma$-diagrams}\label{S:Background}

A diagram is a labeled graph over the vertex set $[n]=\{1, \dots, n\}$ in
which each vertex has degree $\le 3$, represented by drawing its vertices
in a horizontal line and its arcs $(i,j)$, where $i<j$, are
drawn in the upper half-plane. The backbone of a diagram is the sequence of
consecutive integers $(1,\dots,n)$ together with the edges $\{\{i,i+1\}
\mid 1\le i\le n-1\}$. We shall distinguish the backbone edge
$\{i,i+1\}$ from the arc $(i,i+1)$, which we refer to as a $1$-arc.
A stack of length $\tau$ is a maximal sequence of ``parallel'' arcs,
$((i,j),(i+1,j-1),\dots,(i+\tau-1,j-\tau+1))$.
A stack of length $\ge \tau$ is called a $\tau$-canonical stack.
In particular, a stack of length one is an isolated arc.

The specific drawing of a diagram $G$ in the plane determines a cyclic
ordering on the half-edges of the underlying graph incident on each vertex,
thus defining a corresponding fatgraph $\mathbb{G}$. The collection of
cyclic orderings is called fattening, one such ordering on the half-edges
incident on each vertex.
Each fatgraph $\mathbb{G}$ determines an oriented surface $F(\mathbb{G})$
\cite{Loebl:08,Penner:10} which is connected if $G$ is and has some
associated genus $g(G)\ge 0$ and number $r(G)\ge 1$ of boundary components.
Clearly, $F(\mathbb{G})$ contains $G$ as a deformation retract
\cite{Massey:69}.
Fatgraphs were first applied to RNA secondary structures in
\cite{Waterman:93,Penner:03}.

A diagram $G$ hence determines a unique surface $F(\mathbb{G})$
(with boundary). Filling the boundary components with discs we can
pass from $F(\mathbb{G})$ to a surface without boundary. Euler
characteristic, $\chi$, and genus, $g$, of this surface is given by
$$
\chi =  v - e + r\quad\text{\rm and }\quad
g  =  1-\frac{1}{2}\chi,
$$
respectively, where $v,e,r$ is the number of discs, ribbons and boundary
components in $\mathbb{G}$, \cite{Massey:69}.
The genus of a diagram is that of its associated surface without boundary.

A shadow is a diagram without noncrossing arcs, isolated vertices and
stacks of length greater than one. The shadow of a diagram of genus $g$
is obtained by removing all noncrossing arcs, deleting all isolated vertices
and collapsing all induced stacks to single arcs, see Fig.~\ref{F:shadow}.

\begin{figure}
\begin{center}
\includegraphics[width=1\textwidth]{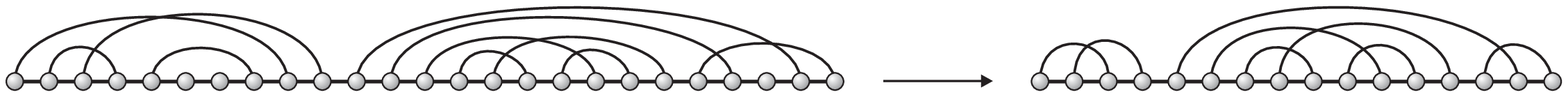}
\end{center}
\caption{\small Shadows: the shadow is obtained by removing all noncrossing
arcs and isolated points and collapsing all stacks and resulting
stacks into single arcs.
}\label{F:shadow}
\end{figure}

The shadow of a diagram $G$, $\sigma(G)$, can possibly
be empty. Furthermore, projecting into the shadow does not affect genus.
Any shadow of genus $g$ over one backbone contains at least $2g$ and at most
$(6g-2)$ arcs. In particular, for fixed genus $g$, there exist only finitely
many shadows \cite{fenix2bb}.

A diagram is called irreducible, if and only if for any two arcs,
$\alpha_1,\alpha_k$ contained in $E$, there exists a sequence of
arcs $(\alpha_1,\alpha_2,\dots,\alpha_{k-1},\alpha_k)$
such that $(\alpha_i,\alpha_{i+1})$ are crossing. Irreducibility
is equivalent to the concept of primitivity introduced by \cite{Bon:08}.
According to \cite{fenix2bb}, for arbitrary genus $g$ and
$2g\le\ell\le (6g-2)$, there exists an irreducible shadow of genus
$g$ having exactly $\ell$ arcs.

The shadow $\sigma(G)$ of a diagram $G$ decomposes into
a set of irreducible shadows. We shall call these shadows irreducible
$G$-shadows. A diagram, $G$, is a $\gamma$-diagram if and only if for any
irreducible $G$-shadow, $G'$, $g(G')\le \gamma$ holds, see Fig.~\ref{F:irrdec}.

\begin{figure}
\begin{center}
\includegraphics[width=0.9\textwidth]{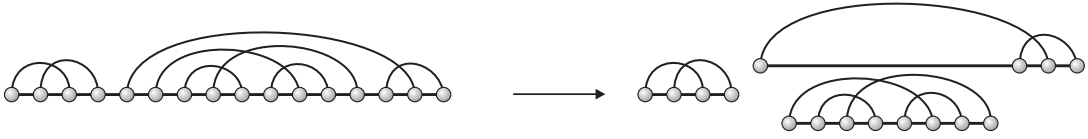}
\end{center}
\caption{\small The shadow $\sigma(G)$ of a diagram $G$ decomposes into
a set of irreducible shadows, which implies that $G$ is a $2$-diagram.
}\label{F:irrdec}
\end{figure}

We call $\tau$-canonical $\gamma$-diagrams without arcs of
the form $(i,i+1)$ ($1$-arcs) $\tau$-canonical
$\gamma$-structures and their set is denoted by $\mathcal{G}_{\tau,\gamma}$.
The set of $\gamma$-diagrams that contain only vertices of degree three
($\gamma$-matchings) is denoted by $\mathcal{H}_{\gamma}$ and the set of
$\gamma$-matchings that contain only stacks of length one ($\gamma$-shapes)
is denoted by $\mathcal{S}_\gamma$.

A stack of length $\tau$,
$((i,j),(i+1,j-1),\dots,(i+\tau-1,j-\tau+1))$ induces a sequence
of pairs $(([i, i+1],[j, j-1]),([i+1,i+2],[j-1,j-2])\dots)$.
We call any of these $2(\tau-1)$ intervals a $P$-interval.
The interval $[i+\tau-1,j-\tau+1]$ is called a $\sigma$-interval,
see Fig.~\ref{F:interval}. We distinguish these two types of interval for special
manipulation in the inflation step.

\begin{figure}
\begin{center}
\includegraphics[width=0.7\textwidth]{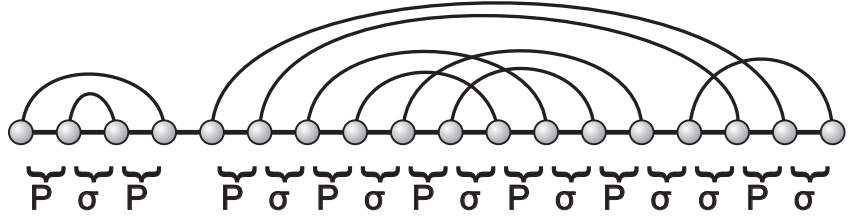}
\end{center}
\caption{\small $\sigma$-intervals and $P$-intervals.
}\label{F:interval}
\end{figure}
\subsection{Some generating functions}

Let ${\bf i}(g,n)$ denote the number of irreducible shadows of genus $g$
with $n$ arcs. Since for fixed genus $g$ there exist only finitely many
shadows we have the generating polynomial of irreducible shadows of genus
$g$
$$
{\bf I}_g(z)=\sum_{n=2g}^{6g-2}\,{\bf i}(g,n)z^n.
$$
For instance, for genus $1$ and $2$ we have
\begin{eqnarray*}
{\bf I}_1(z) &=& z^2  +  2z^3+   z^4,\\
{\bf I}_2(z) &=& 17z^4+160z^5+566z^6+1004z^7+961z^8+476z^9+96z^{10}.
\end{eqnarray*}
We denote the bivariate generating polynomial of irreducible shadows of genus
$\leq \gamma$ by
$$
{\bf Is}_{\gamma}(z,t)=\sum_{g\leq \gamma}\,{\bf I}_{g}(z) t^g.
$$
For example for $\gamma=1$ and $\gamma=2$ we have
\begin{eqnarray*}
{\bf Is}_{1}(z,t)&=&\left( 1+z \right) ^{2}{z}^{2}t,\\
{\bf Is}_{2}(z,t)&=& \left( 1+z \right) ^{4}{z}^{4}
       \left( 24\,z+17 \right)  \left( 4\,z+1
 \right) {t}^{2}+ \left( 1+z \right) ^{2}{z}^{2}t.
\end{eqnarray*}
Let ${\bf h}_{\gamma}(g,n)$ denote the number of $\gamma$-matchings of
genus $g$ with $n$ arcs. The univariate and bivariate generating
functions of $\gamma$-matchings are given by
$$
{\bf H}_{\gamma}(z)=\sum_{n}\sum_{g}\,{\bf h}_{\gamma}(g,n)z^n,\quad
{\bf H}_{\gamma}(z,t)=\sum_{g,n}\,{\bf h}_{\gamma}(g,n) t^g z^n.
$$
Let ${\bf s}_{\gamma}(g,n,m)$ denote the number of $\gamma$-shapes of
genus $g$ with $n$ arcs and $m$ $1$-arcs with generating functions of
$$
{\bf S}_{\gamma}(z,t,e)=\sum_{g,n,m}\,{\bf s}_{\gamma}(g,n,m) t^g z^n e^m.
$$
Finally, ${\bf G}_{\tau,\gamma}(g,n)$ denotes the number of $\tau$-canonical
$\gamma$-structures of genus $g$ with $n$ vertices
with generating function
$$
{\bf G}_{\tau,\gamma}(z,t)=\sum_{g,n}\,{\bf G}_{\tau,\gamma}(g,n) t^g z^n.
$$

\subsection{A central limit theorem}

We next discuss a central limit theorem due to Bender \cite{Bender:73}.
It is proved by analyzing the characteristic function using the
L\'{e}vy-Cram\'{e}r Theorem (Theorem IX.4 in \cite{Flajolet:07a}).
\begin{theorem}\label{T:normal}
Suppose we are given the bivariate generating function
\begin{equation*}
f(z,u)=\sum_{n,t\geq 0}f(n,t)\,z^n\,u^t,
\end{equation*}
where $f(n,t)\geq 0$
and $f(n)=\sum_tf(n,t)$. Let $\mathbb{X}_n$ be a r.v.~such that
$\mathbb{P}(\mathbb{X}_n=t)=f(n,t)/f(n)$. Suppose
\begin{equation*}
[z^n]f(z,e^s)= c(s) \, n^{\alpha}\,  \gamma(s)^{-n}
\left(1+ O\left( \frac{1}{n}\right) \right)
\end{equation*}
uniformly in $s$ in a neighborhood of $0$, where $c(s)$ is continuous
and nonzero near $0$, $\alpha$ is a constant, and $\gamma(s)$ is analytic
near $0$.
Then there exists a pair $(\mu,\sigma)$ such that the normalized
random variable
\begin{equation*}
\mathbb{X}^*_{n}=\frac{\mathbb{X}_{n}- \mu \, n}{\sqrt{{n\,\sigma}^2 }}
\end{equation*}
converges in distribution to a Gaussian variable with a speed of convergence
$O(n^{-\frac{1}{2}})$.
That is we have
\begin{equation*}
\lim_{n\to\infty}
\mathbb{P}\left(
\mathbb{X}^*_{n} < x \right)  =
\frac{1}{\sqrt{2\pi}} \int_{-\infty}^{x}\,e^{-\frac{1}{2}c^2} dc \,
\end{equation*}
where $\mu$ and $\sigma^2$ are given by
\begin{equation*}
\mu= -\frac{\gamma'(0)}{\gamma(0)}
\quad \text{\it and} \quad \sigma^2=
\left(\frac{\gamma'(0)}{\gamma(0)}
\right)^2-\frac{\gamma''(0)}{\gamma(0)}.
\end{equation*}
\end{theorem}


\section{$\gamma$-matchings}\label{S:match}



\subsection{$\gamma$-matchings}

Given a matching $G$, an arc is called maximal if it is maximal with respect
to the partial order
$$
(i,j)\le (i',j') \quad \Longleftrightarrow \quad i'\le i\;\wedge j\le j'.
$$
The arcs-set of a matching $G$ gives rise to a (combinatorial) graph
$\varphi(G)$ obtained by mapping each labeled arc $\alpha$ into the
vertex $\varphi(\alpha)=v_\alpha$ connecting any two such vertices iff
the corresponding arcs are crossing in $G$, $\varphi\colon G \to \varphi(G)$.
A component of a matching $G$ is a set of arcs $A$ such that
$\varphi(A)$ is a component in $\varphi(G)$.
Considering the left- and rightmost endpoints of a component containing
some maximal arc induces a partition of the backbone into subsequent
intervals. $G$ induces over each such interval a sub-matching, to which we
refer to as a block.
By construction, all maximal arcs of a fixed block are contained in a unique
component, see Fig.~\ref{F:blcok}.
\begin{figure}
\begin{center}
\includegraphics[width=0.9\textwidth]{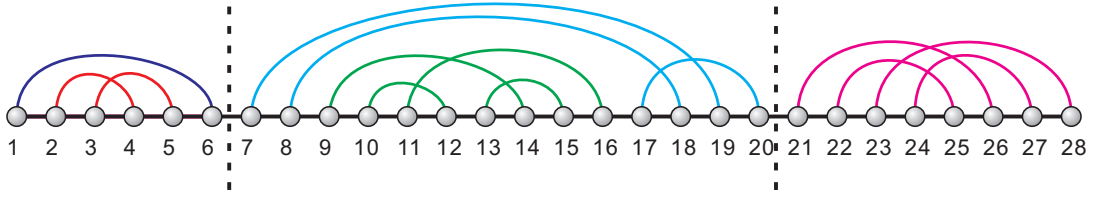}
\end{center}
\caption{\small A $2$-matching $G$ containing the maximal arcs
$(1,6)$, $(7,19)$, $(17,20)$, $(21,26)$, $(23,28)$,
five components and the three  blocks $G[1,6]$, $G[7,20]$, $G[21,28]$.
}\label{F:blcok}
\end{figure}

Any $\gamma$-matching can be decomposed by iteratively removing components
from top to bottom as follows:\\
$\bullet$ one decomposes a $\gamma$-matching into a sequence of blocks\\
$\bullet$ for each block, one removes the unique component containing
all its maximal arcs. \\
Each component can be viewd as a matching by considering it over a backbone.
In this context any such component has genus $\leq \gamma$. By construction
the shadow of a component is always irreducible, see Fig.~\ref{F:dec}.
\begin{figure}
\begin{center}
\includegraphics[width=0.9\textwidth]{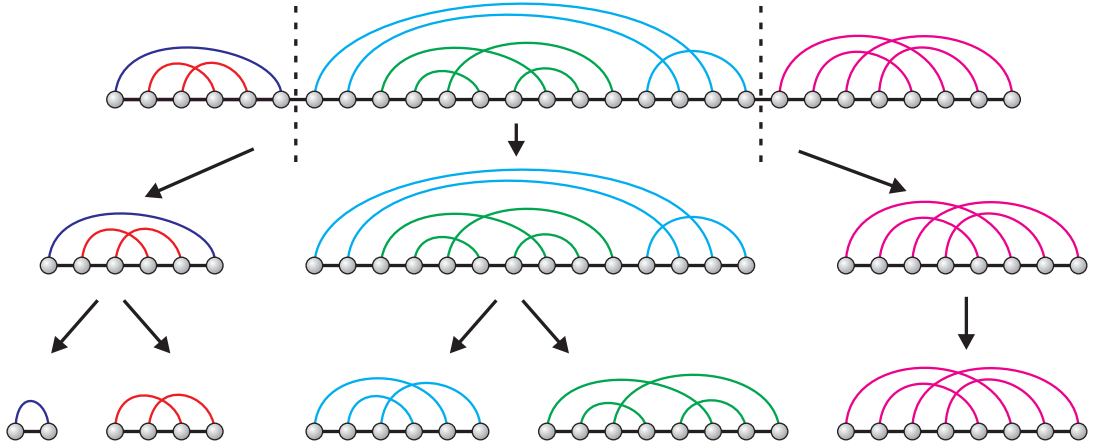}
\end{center}
\caption{\small Decomposition of a $2$-matching by iteratively removing components
from top to bottom.
}\label{F:dec}
\end{figure}

Accordingly, any diagram $G$ can iteratively be decomposed by first removing
all isolated vertices and second by removing components iteratively according
to the above procedure.

The genus of a $\gamma$-matching is additive in the context of the above
decomposition.

\begin{proposition}\label{P:genadd}
Suppose a matching $G$ decomposes into a series of
components $G_1, \ldots , G_n$. Then
$$
g(G)=g(G_1)+\cdots+g(G_n).
$$
\end{proposition}
\begin{proof}
It is suffice to prove the case of a matching $G$ generated by concatenating or nesting two components $G_1$
and $G_2$.\\
Let $n$, $n_1$ and $n_2$ denote the number of arcs in $G$, $G_1$
and $G_2$, respectively.  $r$, $r_1$ and $r_2$ denote the number of boundary components in $G$, $G_1$
and $G_2$, respectively. We have
\begin{equation}\label{E:genusadd1}
\begin{aligned}
2 \,g(G) &= 1+n-r,\\
2 \,g(G_1) &= 1+n_1-r_1,\\
2 \,g(G_2) &= 1+n_2-r_2.
\end{aligned}
\end{equation}
Observe that the following relations hold
\begin{equation}\label{E:genusadd2}
n=n_1+n_2,\quad r=r_1+r_2-1.
\end{equation}
Combining equations~(\ref{E:genusadd1}) and (\ref{E:genusadd2}), we have
\[
g(G)=g(G_1)+g(G_2)
\]
completing the proof.
\end{proof}

\subsection{A functional equation}

In \cite{Han:11a}, the generating function of $\gamma$-matchings has been
computed. In the following we shall refine this generating function by its
inherent genus-filtration.
\begin{theorem}\label{T:Hfuneqn}
Let $R=\mathbb{Z}[z,t]$. Then the following assertions hold:\\
{\bf (a)} the bivariate generating function of $\gamma$-matchings,
${\bf H}_{\gamma}(z,t)$, satisfies
\begin{equation}\label{E:HIrelation1}
{\bf H}_{\gamma}(z,t)^{-1}
 = 1- \left(z\,{\bf H}_{\gamma}(z,t)
 + {\bf H}_{\gamma}(z,t)^{-1}\,
 {\bf Is}_{\gamma}\left(\frac{z\,{\bf H}_{\gamma}(z,t)^2}
 {1-z \,{\bf H}_{\gamma}(z,t)^2},t\right) \right),
\end{equation}
or equivalently,
\begin{equation}\label{E:HIrelation}
{\bf H}_{\gamma}(z,t)-z\,{\bf H}_{\gamma}(z,t)^2-
{\bf Is}_{\gamma}\left(\frac{z\,{\bf H}_{\gamma}(z,t)^2}
 {1-z \,{\bf H}_{\gamma}(z,t)^2},t\right) =1.
\end{equation}
In particular, there exists a polynomial $P_\gamma(z,t,X)\in R[X]$
of degree $(12\gamma -2)$, such that
$P_\gamma(z,t,{\bf H}_{\gamma}(z,t))=0$.\\
{\bf (b)} eq.~(\ref{E:HIrelation}) determines ${\bf H}_{\gamma}(z,t)$
uniquely.
\end{theorem}
\begin{proof}
We distinguish the classes of blocks into two categories characterized by
the unique component containing all maximal arcs (maximal component). Namely,\\
$\bullet$ blocks whose maximal component contains only one arc, \\
$\bullet$ blocks whose maximal component is an (nonempty) irreducible matching.\\
In the first case, the removal of the maximal component (one arc) generates
again a $\gamma$-matching, which translates into the term
$$
z\,{\bf H}_{\gamma}(z,t).
$$
Let ${\bf T}(z,t)$ denote the (genus filtered) generating function of blocks of
the second type. The decomposition of $\gamma$-matchings into a sequence of
blocks implies
\[
{\bf H}_{\gamma}(z,t)^{-1}
 = 1- \left( z\,{\bf H}_{\gamma}(z,t) + {\bf T}(z,t)\right).
\]
Let $\sigma$ be a fixed irreducible shadow of genus $g$ having $n$ arcs. Let
${\bf T}_\sigma(z,t)$ be the generating function of blocks, having $\sigma$ as
the shadow of its unique maximal component. Then we have
\[
{\bf T}(z,t)= \sum_{\sigma \in \mathcal{I}_\gamma}{\bf T}_\sigma(z,t),
\]
where $\mathcal{I}_\gamma$ denotes the set of irreducible shadows of genus
$\leq \gamma$.
We shall construct $\mathcal{T}_\sigma$ in three steps using arcs, $\mathcal{R}$,
sequences of arcs, $\mathcal{K}$, induced arcs, $\mathcal{N}$, sequence of
induced arcs, $\mathcal{M}$, and arbitrary $\gamma$-matching, $\mathcal{H}$.

{\bf Step I:} We inflate each arc in $\sigma$ into a sequence of induced arcs,
see Fig.~\ref{F:H1}. An induced arc, i.e.~an arc together with at least one
nontrivial $\gamma$-matching in either one or in both ${P}$-intervals
\begin{equation*}
\mathcal{N} = \mathcal{R}\times
\left( (\mathcal{H}-1)+(\mathcal{H}-1)+
(\mathcal{H}-1)^2\right)
=\mathcal{R}\times \left(\mathcal{H}^2-1\right).
\end{equation*}
\begin{figure}
\begin{center}
\includegraphics[width=0.8\textwidth]{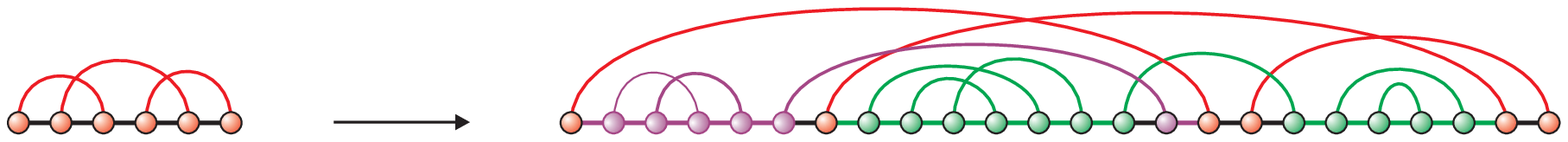}
\end{center}
\caption{\small {\bf Step I:} inflation of each arc in $\sigma$ into a sequence of induced arcs.
}\label{F:H1}
\end{figure}
Clearly, we have for a single induced arc
$\mathbf{N}(z,t)=z \left({\bf H}_{\gamma}(z,t)^2-1 \right)$, guaranteed
by Proposition~\ref{P:genadd},
and for a sequence of induced arcs, $\mathcal{M}=
\textsc{Seq}(\mathcal{N})$, where
\begin{eqnarray*}
{\bf M}(z,t) & = &  \frac{1}{1-z \left({\bf H}_{\gamma}(z,t)^2-1 \right)}.
\end{eqnarray*}
Inflating each arc into a sequence of induced arcs, $R^n \times \mathcal{M}^n$,
gives the corresponding generating function
\[
z^n {\bf M}(z,t)^n=\left(\frac{z}{1-z \left({\bf H}_{\gamma}(z,t)^2-1 \right)}\right)^n,
\]
since, by Proposition~\ref{P:genadd}, the genus is additive.

{\bf Step II:} We inflate each arc in the component with shadow $\sigma$ into
stacks, see Fig.~\ref{F:H2}. The corresponding generating function is
\[
\left(\frac{\frac{z}{1-z}}{1-\frac{z}{1-z}
\left({\bf H}_{\gamma}(z,t)^2-1 \right)}\right)^n=
\left(\frac{z}{1-z {\bf H}_{\gamma}(z,t)^2}\right)^n
\]
\begin{figure}
\begin{center}
\includegraphics[width=0.9\textwidth]{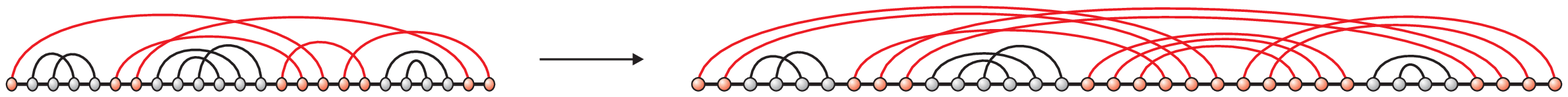}
\end{center}
\caption{\small  {\bf Step II:} inflation of each arc in the component with
shadow $\sigma$ into stacks.
}\label{F:H2}
\end{figure}
{\bf Step III:} We insert additional $\gamma$-matchings at exactly
$(2n-1)$ $\sigma$-intervals, see Fig.~\ref{F:H3}. Accordingly, the
generating function is ${\bf H}_{\gamma}(z,t)^{2n-1}$.
\begin{figure}
\begin{center}
\includegraphics[width=1\textwidth]{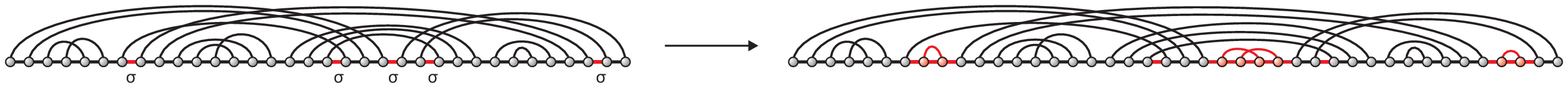}
\end{center}
\caption{\small  {\bf Step III:} insertion of additional $\gamma$-matchings at exactly
$(2n-1)$ $\sigma$-intervals.
}\label{F:H3}
\end{figure}

Combining these three steps and utilizing additivity of the genus, we arrive at
\begin{eqnarray*}
{\bf T}_\sigma(z,t)&=& t^g \left(\frac{z}{1-z {\bf H}_{\gamma}(z,t)^2}\right)^n
{\bf H}_{\gamma}(z,t)^{2n-1}\\
&=& t^g {\bf H}_{\gamma}(z,t)^{-1}\,
      \left(\frac{z {\bf H}_{\gamma}(z,t)^2}{1-z {\bf H}_{\gamma}(z,t)^2}\right)^n.
\end{eqnarray*}
Therefore
\begin{eqnarray*}
{\bf T}(z,t)&=& \sum_{\sigma \in \mathcal{I}_\gamma}{\bf T}_\sigma(z,t)\\
&=& \sum_{g \leq \gamma} {\bf i}(g,n) t^g \,{\bf H}_{\gamma}(z,t)^{-1}\,
\left(\frac{z {\bf H}_{\gamma}(z,t)^2}{1-z {\bf H}_{\gamma}(z,t)^2}\right)^n.
\end{eqnarray*}
We derive
\[
{\bf T}(z,t)={\bf H}_{\gamma}(z,t)^{-1} {\bf Is}
\left(\frac{z\,{\bf H}_{\gamma}(z,t)^2}{1-z\,{\bf H}_{\gamma}(z,t)^2},t\right),
\]
completing the proof of eq.~(\ref{E:HIrelation1}).

Note that ${\bf Is}_{\gamma}(z,t)$ are polynomials in $z$ of degree $6\gamma-2$.
Eq.~(\ref{E:HIrelation}) gives rise to the polynomial
\[
P_{\gamma}(z,t,X)=(1-z X^2)^{6\gamma-2} (-1+X-z X^2)
- (1-z X^2)^{6\gamma-2}  {\bf Is}_{\gamma}\left(\frac{z X^2}{1-z X^2},t\right),
\]
where $deg(P_{\gamma}(z,t,X))= 12\gamma -2$. Then
$P_\gamma(z,t,{\bf H}_{\gamma}(z,t))=0$, whence {\bf (a)}.

It remains to prove {\bf (b)}. Eq.~(\ref{E:HIrelation}) implies
\begin{eqnarray*}
&&(1-z {\bf H}_{\gamma}(z,t)^2)^{6\gamma-2} (-1+{\bf H}_{\gamma}(z,t)-z {\bf H}_{\gamma}(z,t)^2)\\
&&- (1-z {\bf H}_{\gamma}(z,t)^2)^{6\gamma-2}  {\bf Is}_{\gamma}
\left(\frac{z {\bf H}_{\gamma}(z,t)^2}{1-z {\bf H}_{\gamma}(z,t)^2},t\right)=0
\end{eqnarray*}
and consequently
\begin{equation}\label{E:HIrelation2}
 \begin{split}
{\bf H}_{\gamma}(z,t)
&= -{\bf H}_{\gamma}(z,t)\sum_{i=1}^{6\gamma-2}\,
{6\gamma-2 \choose i}
{(-z {\bf H}_{\gamma}(z,t)^2)^i}\\
&+(1+z\,{\bf H}_{\gamma}(z,t)^2)\,
(1-z {\bf H}_{\gamma}(z,t)^2)^{6\gamma-2} \\
&+  (1-z {\bf H}_{\gamma}(z,t)^2)^{6\gamma-2}
{\bf Is}_{\gamma}\left(\frac{z {\bf H}_{\gamma}(z,t)^2}{1-z {\bf H}_{\gamma}(z,t)^2},t\right).
\end{split}
\end{equation}
All coefficients of ${\bf H}_{\gamma}(z,t)$ in the RHS of
eq.~(\ref{E:HIrelation2}), are polynomials in $z$ of
degree $\geq 1$, whence any $[z^n t^g]{\bf H}_{\gamma}(z,t)$ for
$n\ge (6\gamma-1)$ can be recursively computed. Accordingly,
eq.~(\ref{E:HIrelation2}) determines ${\bf H}_{\gamma}(z,t)$ uniquely.
\end{proof}

{\bf Remark:} Proposition~\ref{P:genadd} makes the additional variable marking the genus compatible with the inflation procedure in Theorem~\ref{T:Hfuneqn}.\\
In particular for $\gamma=1$ and $\gamma=2$ we have
\begin{eqnarray*}
P_{1}(z,t,X)&=&-1+X+3\,{X}^{2}z-4\,{X}^{3}z-2\,{X}^{4}{z}^{2}-{X}^{4}t{z}^{2}+6\,{X}^
{5}{z}^{2}\\
&&-2\,{X}^{6}{z}^{3}-4\,{X}^{7}{z}^{3}+3\,{X}^{8}{z}^{4}+{X}^{9}{z}^{4}-{X}^{10}{z}^{5},\\
P_{2}(z,t,X)&=& -1+X+9\,{X}^{2}z-10\,{X}^{3}z-35\,{X}^{4}{z}^{2}-{X}^{4}t{z}^{2}+45\,{X}^{5}{z}^{2}\\
&&+75\,{X}^{6}{z}^{3}+6\,{X}^{6}t{z}^{3}-120\,{X}^{7}{z}^{3
}-90\,{X}^{8}{z}^{4}-15\,{X}^{8}t{z}^{4}-17\,{X}^{8}{t}^{2}{z}^{4}\\
&&+210\,{X}^{9}{z}^{4}+42\,{X}^{10}{z}^{5}+20\,{X}^{10}t{z}^{5}-58\,{X}^{10}{t}^{2}{z}^{5}-252\,{X}^{11}{z}^{5}\\
&&+42\,{X}^{12}{z}^{6}-15\,{X}^{12}t{z}^{6}-21\,{X}^{12}{t}^{2}{z}^{6}+210\,{X}^{13}{z}^{6}-90\,{X}^{14}{z}^{7}\\
&&+6\,{X}^{14}t{z}^{7}-120\,{X}^{15}{z}^{7}+75\,{X}^{16}{z}^{8}-{X}^{16}t{z}^{8}+45\,{X}^{17}{z}^{8}\\
&&-35\,{X}^{18}{z}^{9}-10\,{X}^{19}{z}^{9}+9\,{X}^{20}{z}^{10}+{X}^{21}{z}^{10}-{X}^{22}{z}^{11}.
\end{eqnarray*}



\subsection{Singularity analysis}


The bivariate generating function ${\bf H}_{\gamma}(z,t)$ is not explicitely
known but is completely characterized by the functional equation established
in Theorem~\ref{T:Hfuneqn}.

In the following we employ this implicit equation to obtain key information
about the singular expansion of ${\bf H}_{\gamma}(z,t)$, where we consider
the latter as a univariate generating function parameterized by $t$.
\begin{theorem}\label{T:AsymAlgPra}\cite{Flajolet:07a}
Let $F(z,t)$ be a bivariate function that is analytic at $(0,0)$ and has
non-negative coefficients. Assume that $F(z,t)$ is one of the solutions $y$
of a polynomial equation
\[
\Phi(z,t,y)=0,
\]
where $\Phi$ is a polynomial in $y$, such that $\Phi(z,1,y)$ satisfies the
conditions of Theorem~\ref{T:AsymG}. Define the resultant polynomial
\[
\Delta(z,t)= \mathbf{R} \left(\Phi(z,t,y),
\frac{\partial}{\partial y}\Phi(z,t,y),y \right).
\]
Let $\rho$ be the root of $\Delta(z,1)$, so that $y(z):=F(z,1)$ is
singular at $z =\rho $ and $y(\rho)=\pi$.
Let $\rho(t)$ be the unique root of the equation
\[
\Delta(\rho(t),t)=0,
\]
analytic at $1$, such that $\rho(1)=\rho$.
Then $F(z,t)$ has the singular expansion
\[
F(z,t)= \pi(t) + \lambda(t)\, \left(\rho(t)-z \right)^{\frac{1}{2}} \left(1+ o(1) \right),
\]
where $\pi(t)$ and $\lambda(t)$ are analytic at 1 such that
$\pi(1)=\pi$ and $\lambda(1) \neq 0$. Furthermore
\[
[z^n]F(z,t) = c(t)\, n^{-\frac{3}{2}} \rho(t)^{-n}\left(1+O\left(\frac{1}{n}\right)\right),
\]
uniformly for $t$ restricted to a small neighborhood of $1$, where $c(t)$ is
continuous and nonzero near 1.
\end{theorem}
The following proof is derived from Proposition IX. 17 and Theorem IX. 12 in~\cite{Flajolet:07a}.
\begin{proof}
By Theorem~\ref{T:AsymG}, the function $y(z)=F(z,1)$
has a square-root singularity at $z =\rho$ and admits a singular expansion of the form
\[
F(z,1)=\pi+ \lambda (\rho-z)^{\frac{1}{2}}+O(\rho-z),\quad \text{for some nonzero constant } \lambda.
\]
Singularity analysis then implies the estimate
\[
[z^n]F(z,1) =  c \, n^{-\frac{3}{2}} \rho^{-n} \left( 1+ O\left(\frac{1}{n} \right)\right).
\]
All that is needed now is a uniform lifting of relations above, for $t$ in a small
neighborhood of $1$.

First, the polynomial $\Phi(\rho,1,y)$ has a double
(not triple) zero at $y =\pi$, so that
\[
\left( \frac{\partial}{\partial y}\Phi(\rho,1,y)\right)_{y=\pi} =0 ,\quad \left( \frac{\partial^2}{\partial y^2}\Phi(\rho,1,y)\right)_{y=\pi}\neq 0 .
\]
Thus, the Weierstrass Preparation Theorem gives the local factorization at $(\rho,1,\pi)$
\[
\Phi(z,t,y)= (y^2+ c_1(z,t) y+ c_2(z,t)) \Lambda(z,t,y),
\]
where $\Lambda(z,t,y)$ is analytic and non-zero at $(\rho,1,\pi)$ while $c_1(z,t)$, $c_2(z,t)$ are analytic at $(z,t)=(\rho,1)$.

From the solution of the quadratic equation, we must have locally
\[
y=\frac{1}{2}\left( -c_1(z,t) \pm \sqrt{c_1(z,t)^2- 4 c_2(z,t)}\right).
\]
Consider first $(z,t)$ restricted by  $0 \leq z < \rho $ and $0\leq t <1$. Since $F(z,t)$ is real
there, we must have $c_1(z,t)^2- 4 c_2(z,t)$ also real and non-negative. Since $F(z,t)$ is
continuous and increasing with $z$ for fixed $t$, and since the discriminant $c_1(z,t)^2- 4 c_2(z,t)$ vanishes at $(\rho,1)$,
the minus sign has to be constantly taken. In summary, we have
\[
F(z,t)=\frac{1}{2}\left( -c_1(z,t) - \sqrt{c_1(z,t)^2- 4 c_2(z,t)}\right).
\]
Set $C(z,t):=c_1(z,t)^2- 4 c_2(z,t)$. The function  $C(z,1)$ has a simple real zero
at $z =\rho$. Thus, by the analytic implicit function theorem,
there exists for $t$ sufficiently close to  1,
a unique simple root $\rho(t)$ of the equation $C(\rho(t),t)=0$,
which is an analytic function of $t$ such that $\rho(1)=\rho$.
Set $\widetilde{C}(z,t):= C(z\,\rho(t),t)$, where $C(z,t)$ is analytic at $(\rho,1)$.
Consequently $\widetilde{C}(z,t)$ is analytic at $(1,1)$, since it is a composition of two analytic function. When $t$ sufficiently close to  1, $\widetilde{C}(1,t)=0$ , since $C(\rho(t),t)=0$.
Taking its singular expansion of $\widetilde{C}(z,t)^{\frac{1}{2}}$ at $z=1$,
\[
\widetilde{C}(z,t)^{\frac{1}{2}} =(1-z)^{\frac{1}{2}}\sum_{n \geq 1} \widetilde{C}_n(t) (1-z)^n,
\]
where $\widetilde{C}_n(t)$ is analytic around 1.
For $t \rightarrow 1$, the singular expansion of $C(z,t)^{\frac{1}{2}}$ at $z=\rho(t)$ is given by
\[
C(z,t)^{\frac{1}{2}} =(\rho(t)-z)^{\frac{1}{2}}\sum_{n \geq 1} C_n(t) (\rho(t)-z)^n,
\]
where $C_n(t)$ is analytic around 1.

Then, since $c_1(z,t)$ and $c_2(z,t)$ are analytic, $F(z,t)$ has the singular expansion
\[
F(z,t)= \pi(t) + \lambda(t)\, \left(\rho(t)-z \right)^{\frac{1}{2}} \left(1+ o(1) \right),
\]
where $\pi(t)$ and $\lambda(t)$ are analytic at 1 such that $\pi(1)=\pi$ and $\lambda(1) \neq 0$.
Therefore transfer theorems and the uniformity property of singularity
analysis \cite{Flajolet:07a} imply that
\[
[z^n]F(z,t) = c(t)\, n^{-\frac{3}{2}} \rho(t)^{-n}\left(1+O\left(\frac{1}{n}\right)\right),
\]
uniformly for $t$ restricted to a small neighborhood of $1$, where $c(t)$ is
continuous and nonzero near 1.
\end{proof}

Combining Theorem~\ref{T:Hfuneqn} and Theorem~\ref{T:AsymAlgPra}, we derive

\begin{theorem}\label{T:AsymH}
Let $1\leq \gamma \leq 8$ and $\Delta_{\gamma}(z,s)$ be the resultant of $P_{\gamma}(z,e^s,X)$ and $P_{\gamma}(z,e^s,X)$ as polynomials of $X$.
Let $\rho_{\gamma}(s)$ be the unique root of the equation $\Delta_{\gamma}(\rho_\gamma(s),s)=0$,
analytic at $0$. \\
{\bf (a)} $\rho_{\gamma}(s)$ is the dominant singularity of ${\bf H}_{\gamma}(z,e^s)$, \\
{\bf (b)} then ${\bf H}_{\gamma}(z,e^s)$ has the expansion
\[
{\bf H}_{\gamma}(z,e^s) = \pi_\gamma(s) + \lambda_\gamma(s)\,
\left(\rho_{\gamma}(s) -z \right)^{\frac{1}{2}} \left(1+ o(1) \right),
\]
where $\pi_\gamma(s)$ and $\lambda_\gamma(s)$ are analytic at $0$ such that
$\pi(0)=\pi$ and $\lambda(0) \neq 0$.\\
{\bf (c)} the coefficients of ${\bf H}_{\gamma}(z,e^s) $ are asymptotically given by
\begin{eqnarray}
[z^n]{\bf H}_{\gamma}(z,e^s) &= &
c_\gamma(s) \, n^{-3/2}\, \left(\frac{1}{\rho_{\gamma}(s)}\right)^n
\left(1+ O\left( \frac{1}{n}\right) \right),
\end{eqnarray}
uniformly for $s$ restricted to a small neighborhood of $0$, where $c_\gamma(s)$ is
continuous and nonzero near 0.
\end{theorem}
\begin{proof}
By Theorem~\ref{T:Hfuneqn}, ${\bf H}_{\gamma}(z,e^s)$ satisfies the
algebraic equation $P_{\gamma}(z,e^s,X)=0$.
Theorem~\ref{T:HUniAsy} implies that for $1\leq \gamma \leq 8$,
${\bf H}_{\gamma}(z,1)$ satisfies the conditions required by
Theorem~\ref{T:AsymG}.
This puts us in position to apply Theorem~\ref{T:AsymAlgPra}, from
which the theorem then follows.
\end{proof}


\section{$\gamma$-structures}


\subsection{Combinatorics of $\gamma$-structures}


\begin{lemma}\label{L:oo}
For any $\gamma\geq 1$, we have
\begin{eqnarray}
{\bf S}_\gamma(z,t,e) & = & \frac{1+z}{1+2z-z e}
                    {\bf H}_\gamma\left(\frac{z(1+z)}{(1+2z-z e)^2},t\right).
\end{eqnarray}
\end{lemma}
\begin{proof}
Note that collapsing the stacks, adding or deleting $1$-arcs do not change the genus. Therefore, we can extend
the function equation of Lemma 3 in \cite{Han:11a} to the bivariate case with parameter $t$ marking the genus.
\end{proof}

Using symbolic methods we can conclude from Lemma~\ref{L:oo}
\begin{lemma}
Let $\lambda$ be a fixed $\gamma$-shape of genus $g$ with $s\geq 1$ arcs and
$m\geq 0$ 1-arcs. Then the generating function of $\tau$-canonical
$\gamma$-diagrams containing no $1$-arc that have shape $\lambda$ is given by
$$
{\bf G}^{\lambda}_{\tau,\gamma}(z,t)=(1-z)^{-1}\left(\frac{z^{2\tau}}
{(1-z^2)(1-z)^2-(2z-z^2)z^{2\tau}}\right)^s \, z^m \, t^g.$$
In particular, ${\bf G}^\lambda_{\tau,\gamma} (z,t)$ depends only upon the genus, the
number of arcs and $1$-arcs in $\lambda$.
\end{lemma}

The generating function of $\tau$ canonical $\gamma$-structures now follows:

\begin{theorem}\label{T:GHrelation}
Suppose $\gamma,\tau\geq 1$  and let  $u_\tau(z)
=\frac{(z^2)^{\tau-1}}{z^{2\tau}-z^2+1}$. Then the generating function
${\bf G}_{\tau,\gamma}(z,t)$ is algebraic and given by
\begin{eqnarray}
{\bf G}_{\tau,\gamma}(z,t) & = & \frac{1}{u_\tau(z) z^2-z+1}\
                            {\bf H}_\gamma\left(\frac{u_\tau(z)z^2}
                            {\left(u_\tau(z) z^{2}-z+1\right)^2},t\right).
\end{eqnarray}
\end{theorem}
As for the proof of Theorem~\ref{T:GHrelation},
Proposition~\ref{P:genadd} guarantees that the topological genus
is soley generated by the crossing pattern of the components and
not affected by inflation or the adding of vertices. It is therefore
straightforward to extend the functional equation established in
Theorem $3$ in \cite{Han:11a} to the bivariate case.


\subsection{The genus distribution of $\gamma$-structures}


In this section we study the random variable $X_{n,\tau,\gamma}$
having the distribution
\[
\mathbb{P}(X_{n,\tau,\gamma}=g)=\frac{{\bf G}_{\tau,\gamma}(g,n)}{{\bf G}_{\tau,\gamma}(n)},
\]
where $g=0,1,\ldots,\lfloor \frac{n}{2}\rfloor$.
We shall prove the following
\begin{theorem}\label{T:kk}
There exist a pair $(\mu_{\tau,\gamma},\sigma_{\tau,\gamma})$ such that the normalized random
variable
\begin{equation*}
Y_{n,\tau,\gamma}=\frac{X_{n,\tau,\gamma}- \mu_{\tau,\gamma} \, n}{\sqrt{{n\,
\sigma_{\tau,\gamma}}^2 }}
\end{equation*}
converges in distribution to a Gaussian variable with a speed of convergence $O(n^{-\frac{1}{2}})$. That
is we have
\begin{equation*}
\lim_{n\to\infty}\mathbb{P}\left(\frac{X_{n,\tau,\gamma}- \mu_{\tau,\gamma}
n}{\sqrt{n\, \sigma_{\tau,\gamma}^2}} < x \right)  =
\frac{1}{\sqrt{2\pi}}\int_{-\infty}^{x}\,e^{-\frac{1}{2}c^2} dc \ ,
\end{equation*}
where $\mu_{\tau,\gamma}$ and $\sigma_{\tau,\gamma}^2$ are given by
\begin{equation}\label{E:cltval}
\mu_{\tau,\gamma}= -\frac{\theta_{\tau,\gamma}'(0)}{\theta_{\tau,\gamma}(0)},
\qquad \qquad \sigma_{\tau,\gamma}^2=
\left(\frac{\theta_{\tau,\gamma}'(0)}{\theta_{\tau,\gamma}(0)}
\right)^2-\frac{\theta_{\tau,\gamma}''(0)}{\theta_{\tau,\gamma}(0)}.
\end{equation}
\end{theorem}
In Tables~\ref{Tab:genus}, we present the values of the pairs $(\mu_{\tau,\gamma},\sigma_{\tau,\gamma})$.
\begin{table}[htbp]
\begin{center}
\begin{tabular}{ccccccc}
\hline
&\multicolumn{2}{c}{$\tau=1$} &\multicolumn{2}{c}{$\tau=2$}
&\multicolumn{2}{c}{$\tau=3$}  \\
\hline
& $\mu_{\tau,\gamma}$ & $\sigma_{\tau,\gamma}^2$
& $\mu_{\tau,\gamma}$ & $\sigma_{\tau,\gamma}^2$
& $\mu_{\tau,\gamma}$ & $\sigma_{\tau,\gamma}^2$\\
\hline
$\gamma=1$ & 0.091240 & 0.021067 & 0.041235 & 0.009358 & 0.026632 & 0.006043\\
$\gamma=2$ & 0.112037 & 0.022088 & 0.050436 & 0.009768 & 0.032564 & 0.006288 \\
\hline
&\multicolumn{2}{c}{$\tau=4$}
&\multicolumn{2}{c}{$\tau=5$} &\multicolumn{2}{c}{$\tau=6$}\\
\hline
& $\mu_{\tau,\gamma}$ & $\sigma_{\tau,\gamma}^2$
& $\mu_{\tau,\gamma}$ & $\sigma_{\tau,\gamma}^2$
& $\mu_{\tau,\gamma}$ & $\sigma_{\tau,\gamma}^2$ \\
\hline
$\gamma=1$ & 0.019706 & 0.004481 & 0.015666 & 0.003571 & 0.013017 & 0.002974\\
$\gamma=2$ & 0.024104 & 0.004657 & 0.019170 & 0.003709 & 0.015935 & 0.003087\\
\hline
\end{tabular}
\centerline{}
\smallskip
\caption{\small {\sf Genus distribution:}
The central limit theorem for the genus in $\gamma$-structures.
We list $\mu_{\tau,\gamma}$ and $\sigma^2_{\tau,\gamma}$ derived from eq.~(\ref{E:cltval}).
}
\label{Tab:genus}
\end{center}
\end{table}

Theorem~\ref{T:kk} follows from Theorem~\ref{T:normal} setting
$$
f(z,e^s)={\bf G}_{\tau,\gamma}(z,e^s)
$$
and we shall subsequently verify the applicability of the latter.

The crucial prerequisite for applying Theorem~\ref{T:normal} is accomplished by
Theorem~\ref{T:rr} which in turn is implied by Theorem~\ref{T:uniform} below, which
guarantees
\begin{equation}\label{E:ll}
[z^n]{\bf H}_\gamma(\psi(z),e^s) =A(s)\,n^{-\frac{3}{2}}
\left(\frac{1}{\kappa(s)}\right)^n \left(1+ O\left( \frac{1}{n}\right) \right),\quad
\text{$A(s)$ continuous},
\end{equation}
for $1\leq \gamma\leq 7$. Once eq~(\ref{E:ll}) is established,
the analyticity of
$\kappa(s)$ is guaranteed by the analytic implicit function theorem \cite{Flajolet:07a}.

Note that Theorem~\ref{T:AsymH} already guarantees that
the coefficients of ${\bf H}_{\gamma}(z,e^s) $ are asymptotically given by
\begin{eqnarray}
[z^n]{\bf H}_{\gamma}(z,e^s) &= &
c_\gamma(s) \, n^{-3/2}\, \left(\frac{1}{\rho_{\gamma}(s)}\right)^n
\left(1+ O\left( \frac{1}{n}\right) \right),
\end{eqnarray}
uniformly for $s$ restricted to a small neighborhood of $0$, where $c_\gamma(s)$ is
continuous and nonzero near 0.
However, according to Theorem~\ref{T:GHrelation} we have
\begin{eqnarray*}
{\bf G}_{\tau,\gamma}(z,t) & = & \frac{1}{u_\tau(z) z^2-z+1}\
                            {\bf H}_\gamma\left(\frac{u_\tau(z)z^2}
                            {\left(u_\tau(z) z^{2}-z+1\right)^2},t\right).
\end{eqnarray*}
Consequently we have to establish uniform convergence for generating functions
of the form ${\bf H}_\gamma(\psi(z),e^s)$, where $\psi(z)$ is analytic for
for $|z|<r$.

\begin{theorem}\label{T:uniform}
Suppose $1\leq \gamma\leq 7$. Let $\psi(z)$ be an analytic function for $|z|<r$,
such that $\psi(0)=0$. In addition suppose $\kappa(s)$ is the
unique dominant singularity of ${\bf H}_\gamma(\psi(z),e^s)$ and analytic
solution of $\psi(\kappa(s))=\rho_\gamma(s)$, $|\kappa(s)|\leq r$,
$\frac{d}{d z}\psi(\kappa(s))\neq 0$ for $\vert s\vert<\epsilon$.
Then ${\bf H}_\gamma(\psi(z),e^s)$ has a singular expansion and
\begin{equation*}
[z^n]{\bf H}_\gamma(\psi(z),e^s) = A(s)\,n^{-\frac{3}{2}}
\left(\frac{1}{\kappa(s)}\right)^n\left(1+ O\left( \frac{1}{n}\right) \right)
\quad \text{for some continuous
$A(s)\in\mathbb{C}$},
\end{equation*}
uniformly in $s$ contained in a small neighborhood of $0$.
\end{theorem}
We prove Theorem~\ref{T:uniform} in Section~\ref{S:appendix}.

We proceed by applying Theorem~\ref{T:uniform} in order to derive an asymptotic
formula for the coefficients of ${\bf G}_{\tau,\gamma}(z,e^s)$ viewed as a univariate
generating function, parameterized by $e^s$. The key point here is that this formula
is uniform in the parameter $s$, close to $0$.

\begin{theorem}\label{T:rr}
For $1 \leq \gamma \leq 8$ and $1\leq \tau \leq 10$, ${\bf G}_{\tau,\gamma}(z,e^s)$
has a unique dominant singularity, $\theta_{\tau,\gamma}(s)$, such that for $s$
restricted to a small neighborhood of $0$:\\
{\bf (1)} $\theta_{\tau,\gamma}(s)$ is analytic,\\
{\bf (2)} $\theta_{\tau,\gamma}(s)$ is the solution of minimal modulus of
\[
\frac{u_\tau(z)z^2}{\left(u_\tau(z) z^{2}-z+1\right)^2}= \rho_{\gamma}(s)
\]
{\bf(3)}
\begin{equation*}
[z^n]{\bf G}_{\tau,\gamma}(z,e^s)  =   k_{\tau,\gamma}(s) \,n^{-\frac{3}{2}} \left(\frac{1}{\theta_{\tau,\gamma}(s)}\right)^n\left(1+ O\left( \frac{1}{n}\right) \right)
\end{equation*}
uniformly for $s$ restricted to a small neighborhood of $0$, where $k_{\tau,\gamma}(s)$
is continuous and nonzero near $1$.
\end{theorem}

\begin{proof}
The first step is to establish the existence and uniqueness of the
dominant singularity $\theta_{\tau,\gamma}(s)$.\\
We denote
\begin{eqnarray}
\vartheta_\tau(z)&=&u_\tau(z) z^{2}-z+1,\\
\psi_{\tau}(z)&=&\frac{u_\tau(z)z^2}{\left(u_\tau(z) z^{2}-z+1\right)^2},
\end{eqnarray}
and consider the equations
\begin{equation*}
\forall \, 1\leq \gamma \leq 8;\qquad F_{\tau,\gamma}(z,s)=\psi_{\tau}(z)-\rho_\gamma(s),
\end{equation*}
where $\rho_\gamma(s)$ is defined in Theorem~\ref{T:AsymH}.
Theorem~\ref{T:GHrelation} and Theorem~\ref{T:uniform} imply that the
singularities of ${\bf G}_{\tau,\gamma}(z,e^s)$ are
are contained in the set of roots of
\begin{equation*}
F_{\tau,\gamma}(z,s)=0\quad\text{and}\quad \vartheta_\tau(z)=0
\end{equation*}
where $1\leq \gamma \leq 8$.
Let $\theta_{\tau,\gamma}$ denote the solution of minimal modulus of
\begin{equation*}
F_{\tau,\gamma}(z,0)=\psi_{\tau}(z)-\rho_\gamma(0)=0.
\end{equation*}
We next verify that, for sufficiently small $\epsilon_\gamma>0$,
$|z-\theta_{\tau,\gamma}|<\epsilon_\gamma$, $|s|<\epsilon_\gamma$, the following
assertions hold
\begin{itemize}
\item $\frac{\partial}{\partial z} F_{\tau,\gamma}(\theta_{\tau,\gamma},0)\neq 0$
\item $\frac{\partial}{\partial z} F_{\tau,\gamma}(z,s)$ and
      $\frac{\partial}{\partial s} F_{\tau,\gamma}(z,s)$ are
      continuous.
\end{itemize}
The analytic implicit function theorem, guarantees the existence of a unique
analytic function $\theta_{\tau,\gamma}(s)$ such that, for $|s|<\epsilon_i$,
\begin{equation*}
 F_{\tau,\gamma}(\theta_{\tau,\gamma}(s),s)=0\quad\text{ and }\quad
\theta_{\tau,\gamma}(0)=\theta_{\tau,\gamma}.
\end{equation*}
Analogously, we obtain the minimal solution  $\delta_\tau$ of
$\vartheta_\tau(z)=0$.
We next verify that the unique dominant singularity of
${\bf G}_{\tau,\gamma}(z,1)$ is the minimal positive solution $\theta_{\tau,\gamma}$ of
$F_{\tau,\gamma}(z,0)=0$ and subsequently using an continuity argument.
Therefore, for sufficiently small $\epsilon$ where
$\epsilon<\epsilon_i$,
$|s|<\epsilon$, the modulus of $\theta_{\tau,\gamma}(s)$, for
$1\leq \gamma \leq 8$
and $\delta_\tau$ are all strictly larger than the modulus of
$\theta_{\tau,\gamma}(s)$. Consequently, $\theta_{\tau,\gamma}(s)$ is
the unique dominant singularity of ${\bf G}_{\tau,\gamma}(z,e^s)$.\\
{\it Claim.} There exists some continuous $ k_{\tau,\gamma}(s)$ such that, uniformly in
$s$, for $s$ in a neighborhood of $0$
\begin{equation*}
[z^n]{\bf G}_{\tau,\gamma}(z,e^s) =   k_{\tau,\gamma}(s) \,n^{-\frac{3}{2}}
\left(\frac{1}{\theta_{\tau,\gamma}(s)}\right)^n\left(1+ O\left( \frac{1}{n}\right) \right).
\end{equation*}
To prove the Claim, let $r$ be some positive real number such that
$\theta_{\tau,\gamma}<r<\delta_\tau$.
For sufficiently small $\epsilon>0$ and $|s|<\epsilon$,
$$
|\theta_{\tau,\gamma}(s)|\leq r.
$$
Then $\psi_{\tau}(z)$ and $\frac{1}{\vartheta_\tau(z)}$ are all analytic in
$|z|\leq r$ and $\psi_{\tau}(0)=0$.
Since $\theta_{\tau,\gamma}(s)$ is the unique dominant singularity of
\begin{equation*}
{\bf G}_{\tau,\gamma}(z,e^s) =\frac{1}{\vartheta_\tau(z)}
\,{\bf H}_\gamma\left(\psi_{\tau}(z),t\right),
\end{equation*}
satisfying
\begin{equation*}
\psi_{\tau}(\theta_{\tau,\gamma}(s))=\rho_\gamma(s)\quad\text{and}\quad
|\theta_{\tau,\gamma}(s)|\leq r,
\end{equation*}
for $|s|<\epsilon$. For sufficiently small $\epsilon>0$,
$\frac{\partial}{\partial z} F_{\tau,\gamma}(z,s)$ is continuous and
$\frac{\partial}{\partial z}F_{\tau,\gamma}(\theta_{\tau,\gamma},0)\neq 0$.
Thus there exists some $\epsilon>0$, such that for $|s|<\epsilon$,
$\frac{\partial}{\partial z} F_{\tau,\gamma}(\theta_{\tau,\gamma}(s),s)\neq 0$.
According to Theorem~\ref{T:uniform}, we therefore derive
\begin{equation*}
[z^n]{\bf G}_{\tau,\gamma}(z,e^s) =   k_{\tau,\gamma}(s) \,n^{-\frac{3}{2}} \left(\frac{1}{\theta_{\tau,\gamma}(s)}\right)^n\left(1+ O\left( \frac{1}{n}\right) \right),
\end{equation*}
uniformly for $s$ restricted to a small neighborhood of $0$, where $k_{\tau,\gamma}(s)$
is continuous and nonzero near $1$.

\end{proof}



\section{Discussion}


Our results trigger a series of intriguing research perspectives. The genus filtration
and in particular the emergence of the bivariate polynomials ${\bf Is}_g(z,t)$ gives
rise to the question whether we can find bijective, constructive proofs in order
to establish recurrences with respect to the topological genus $g$. It would be
fascinating to be able to construct genus $(\gamma+1)$-structures from lower genera.
Such genus-recurrences could have profound algorithmic impact and be of great
practical value.

It is furthermore now clear how to introduce a genus filtration into
$\gamma$-interaction structures \cite{qinjing}. A result of \cite{fenix2bb}
indicates how this can be derived. There it is
proved how to compute the topological genus of a $\gamma$-interaction structure.
The latter formula is in difference to the case of a single backbone not simply
the sum of the genera of its irreducible shadows. This computation can be weaved
into the combinatorial construction presented here in order to refine our
results by the topological genus.

Our analysis of the genus-distribution in $\gamma$-structures shows how the
minimum stack size in these structures affects the expected genus. While for
canonical $1$-structures of length $100$ we have an expected genus of $4$ and
this drops to $2.6$ when requiring a minimum stack-size of $3$,
even to $2.0$ for a minimum stack-size of $4$,  see Fig.~\ref{F:lo}.
\begin{figure}
\begin{center}
\includegraphics[width=0.4\textwidth]{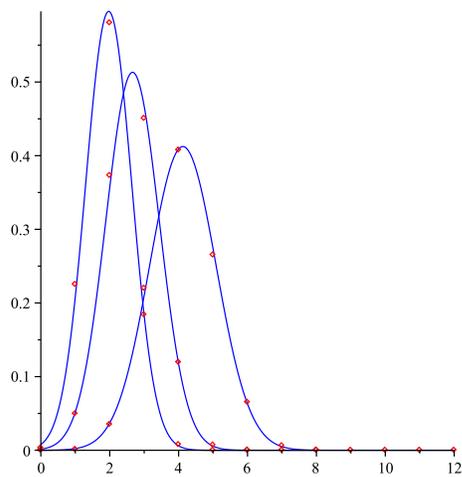}
\end{center}
\caption{\small The effect of minimum stack-size:
the shift of the central limit distribution of
topological genera of $2$-canonical, $3$-canonical and
$4$-canonical $1$-structures for $n=100$.
}\label{F:lo}
\end{figure}
In view of the fact that natural RNA
structures are typically low energy structures and energy is dominated by the
stacking of adjacent base pairs and not by the hydrogen bonds of the individual
base pairs \cite{Mathews:99}, as well as minimum arc-length conditions
\cite{Waterman:78aa},
our results provide some insight why relatively low genera are being
observed in natural RNA structures.


\section{Appendix}\label{S:appendix}



\subsection{Asymptotic analysis of $\gamma$-matchings}

In this section we establish several results employed in the course of
Section~\ref{S:match}. In \cite{Han:11a}, the singular expansion of
${\bf H}_{\gamma}(z)$ has been computed using the method of Newton polygons
and the Newton-Puiseux expansion has been derived.

The following result of \cite{Flajolet:07a} is based on the same arguments but
obsoletes the irreducibility of the polynomials representing the algebraic
equations. As a result, we can compute singular expansions for higher
$\gamma$.

\begin{theorem}\label{T:AsymG}
Let $y(z)=\sum_{n\geq 0}y_n z^n$ be a generating function, analytic at $0$, satisfy
a polynomial equation $\Phi(z,y)=0$. Let $\rho$ be the real dominant singularity
of $y(z)$. Define the resultant of $\Phi(z,y)$ and $\frac{\partial}{\partial y}
\Phi(z,y)$ as polynomial in $y$
\[
\Delta(z)= \mathbf{R} \left(\Phi(z,y), \frac{\partial}{\partial y}\Phi(z,y),y
\right).
\]
{\bf(1)} The dominant singularity $\rho$ is unique and a root of the
resultant $\Delta(z)$ and there exists $\pi=y(\rho)$, satisfying the system of
equations,
\begin{equation}\label{E:phi1}
\Phi(\rho,\pi)=0,\quad \Phi_y(\rho,\pi)=0.
\end{equation}
{\bf(2)} If $\Phi(z,y)$ satisfies the conditions:
\begin{equation}\label{E:phi2}
\Phi_z(\rho,\pi)\neq 0,\quad \Phi_{y y}(\rho,\pi)\neq 0,
\end{equation}
then $y(z)$ has the following expansion at $\rho$
\begin{equation}\label{E:asym1}
y(z)=\pi+ \lambda (\rho-z)^{\frac{1}{2}}+O(\rho-z),\quad \text{for some nonzero
constant } \lambda.
\end{equation}
Further the coefficients of $y(z)$ satisfy
\[
[z^n]y(z) \sim  c \, n^{-\frac{3}{2}} \rho^{-n}, \quad n\rightarrow \infty,
\]
for some constant $c>0$.
\end{theorem}
\begin{proof}
The proof of {\bf(1)} can be found in \cite{Flajolet:07a} or \cite{Hille:62}
pp. 103.
To prove {\bf(2)}, let $\Psi(z,y)=\Phi(\rho-z,\pi-y)$. Immediately, we have
$\Psi(0,0)=0$. Puiseux's Theorem~\cite{Wall:04} guarantees a solution of
$y-\pi$ in terms of a power series in fractional powers of $\rho-z$.
Note that equations~(\ref{E:phi1}) and (\ref{E:phi2}) are equivalent to
\[
\Psi(0,0)=0,\quad \Psi_y(0,0)=0,\quad \Psi_z(0,0)\neq 0,\quad \Psi_{y y}(0,0)\neq 0.
\]
Then we apply Newton's polygon method to determine the type of expansion and
find the first exponent of $z$ to be $\frac{1}{2}$. Therefore $y(z)$ has the
required form of singular expansion. The asymptotics of the coefficients follows
from eq.~(\ref{E:asym1}) as a straightforward application of the transfer
theorem (\cite{Flajolet:07a}, pp. 389 Theorem VI.3).
\end{proof}

Combining Theorem 1 in \cite{Han:11a} and Theorem~\ref{T:AsymG}, the asymptotic
analysis of ${\bf H}_{\gamma}(z)$ follows, generalizing the results in \cite{Han:11a}.
\begin{theorem}\label{T:HUniAsy}
For $1\leq \gamma \leq 8$, let
\[
\Delta_\gamma(z)= \mathbf{R} \left(P_\gamma(z,X), \frac{\partial}{\partial X}P_\gamma(z,X),X \right)
\]
the resultant of $P_\gamma(z,X)$ and $\frac{\partial}{\partial X}P_\gamma(z,X)$ as polynomials in $X$,
and $\rho_{\gamma}$ denote the real dominant singularity of ${\bf H}_{\gamma}(z)$. \\
{\bf (a)} the dominant singularity $\rho_{\gamma}$ is unique and a root of $\Delta_\gamma(z)$, \\
{\bf (b)} at $\rho_{\gamma}$ we have
\begin{equation*}
{\bf H}_{\gamma}(z) = \pi_\gamma  +  \lambda_\gamma (\rho_{\gamma}-z)^{\frac{1}{2}}+O(\rho_{\gamma}-z),\quad \text{for some nonzero constant } \lambda_{\gamma}.
\end{equation*}
{\bf (c)} the coefficients of ${\bf H}_{\gamma}(z)$ are asymptotically given by
\begin{eqnarray*}
[z^n]{\bf H}_{\gamma}(z) & \sim &  c_\gamma \, n^{-3/2}\,\rho_{\gamma}^{-n}
\end{eqnarray*}
for some $c_\gamma>0$.
\end{theorem}
\begin{proof}
Pringsheim¡¯s Theorem (\cite{Flajolet:07a} pp. 240) guarantees that for any $\gamma$, ${\bf H}_{\gamma}(z)$ has a
dominant real singularity $\rho_{\gamma}>0$.
To prove the singular expansion of the function and asymptotic of the coefficients, we verify $P_\gamma(z,X)$ satisfy the condition of Theorem~\ref{T:AsymG} and the results follow.
\end{proof}


\subsection{Proof of Theorem~\ref{T:uniform}}
\begin{proof}
We consider the composite function ${\bf H}_\gamma(\psi(z),e^s)$. In view of
$$
[z^n]f(z,s)=\gamma^n [z^n]f(\frac{z}{\gamma},s),
$$ it suffices to analyze
the function ${\bf H}_\gamma(\psi(\kappa(s) z),e^s)$ and to subsequently rescale
in order to obtain the correct exponential factor. For this purpose we set
\begin{equation*}
\widetilde{\psi}(z,s)=\psi(\kappa(s)z),
\end{equation*}
where $\psi(z)$ is analytic in $|z|<r$. Consequently
$\widetilde{\psi}(z,s)$ is analytic in $|z|<\widetilde{r}$ and
$|s|<\widetilde{\epsilon}$, for some
$1<\widetilde{r},\,0<\widetilde{\epsilon}<\epsilon$,
since it's a composition
of two analytic functions. Taking its Taylor
expansion at $z=1$,
\begin{equation}\label{E:taylor}
\widetilde{\psi}(z,s)=\sum_{n\geq0}\widetilde{\psi}_n(s)(1-z)^n,
\end{equation}
where $\widetilde{\psi}_n(s)$ is analytic in $|s|<\widetilde{\epsilon}$.
The singular expansion of ${\bf H}_\gamma(\psi(z),e^s)$, $1\leq \gamma\leq 8$,
for $z\rightarrow \rho_\gamma(s)$, follows from Theorem~\ref{T:AsymH},
and is given by
\begin{equation*}
 {\bf H}_{\gamma}(z,e^s) = \pi_\gamma(s) + \lambda_\gamma(s)\, \left(\rho_{\gamma}(s) -z \right)^{\frac{1}{2}} \left(1+ o(1) \right).
\end{equation*}
By assumption, $\kappa(s)$ is the unique analytic solution of
$\psi(\kappa(s))=\rho_\gamma(s)$, for $|\kappa(s)|\leq r$, and by construction
${\bf H}_\gamma(\psi(\kappa(s) z),e^s)={\bf H}_\gamma(\widetilde{\psi}(z,s),e^s)$.
In view of eq.~(\ref{E:taylor}), we have for $z\rightarrow 1$ the
expansion
\begin{equation}\label{E:uinner}
\widetilde{\psi}(z,s)-\rho_{\gamma}(s)
=\sum_{n\geq1}\widetilde{\psi}_n(s)(1-z)^n= \widetilde{\psi}_1(s)(1-z)(1+o(1)),
\end{equation}
that is uniform in $s$ since $\widetilde{\psi}_n(s)$ is analytic for
$|s|<\widetilde{\epsilon}$ and
$\widetilde{\psi}_0(s)=\psi(\kappa(s))=\rho_{\gamma}(s)$.

As for the singular expansion of ${\bf H}_\gamma(\widetilde{\psi}(z,s),e^s)$
we derive, substituting the eq.~(\ref{E:uinner})
into the singular expansion of $ {\bf H}_{\gamma}(z,e^s))$, for $z\rightarrow 1$,
\begin{equation*}
\pi_\gamma(s) + c_\gamma (s)\, \left(1 -z \right)^{\frac{1}{2}} \left(1+ o(1)
\right),
\end{equation*}
where $c_\gamma (s) =\lambda_\gamma(s)(- \widetilde{\psi}_1(s))^{\frac{1}{2}}$ and
\begin{equation*}
\widetilde{\psi}_1(s)=\partial_z\widetilde{\psi}(z,s)|_{z=1}
=\kappa(s)\frac{d}{d z}\psi(\kappa(s))\neq 0 \quad \text{\rm for }
\vert s\vert<\epsilon.
\end{equation*}
Furthermore $\pi_\gamma(s) $ is analytic at
$|z|\leq 1$, whence $[z^n]\pi_\gamma(s) $ is exponentially small
compared to $1$. Therefore we arrive at
\begin{equation}\label{E:coeffbiva}
[z^n]{\bf H}_\gamma(\widetilde{\psi}(z,s),e^s)= [z^n] c_\gamma (s)\, \left(1 -z \right)^{\frac{1}{2}} \left(1+ o(1) \right)
\end{equation}
uniformly in $|s|<\widetilde{\epsilon}$.
We observe that $c_\gamma (s)$ is analytic in $|s|<\widetilde{\epsilon}$.
Note that a dependency in the parameter $s$ is only given in the
coefficients $c_k(s)$, that are analytic in $s$.
Standard transfer theorems \cite{Flajolet:07a} imply that
\begin{equation*}
[z^n]{\bf H}_\gamma(\widetilde{\psi}(z,s),e^s) =
A(s)\,n^{-\frac{3}{2}}\left(1+ O\left( \frac{1}{n}\right) \right) \quad \text{\rm for some
$A(s)\in\mathbb{C}$},
\end{equation*}
uniformly in $s$ contained in a small neighborhood of $0$.
Finally, as mention in the beginning of the proof, we use the scaling property
of Taylor expansions in order to derive
\begin{equation*}
[z^n]{\bf H}_\gamma(\psi(z),e^s) =\left(\kappa(s)\right)^{-n}[z^n]{\bf H}_\gamma(\widetilde{\psi}(z,s),e^s)
\end{equation*}
and the proof of the Theorem is complete.
\end{proof}

\bibliographystyle{elsarticle-num}

\begin{thebibliography}{10}

\bibitem{Bailor:10}
M.H. Bailor, X. Sun, H.M. Al-Hashimi, Topology links RNA secondary structure
with global conformation, dynamics, and adaptation,
Science 327 (2010) 202--206.

\bibitem{Tabaska:98}
J.E. Tabaska, R.B. Cary, H.N. Gabow, G.D. Stormo, An RNA folding method capable of
identifying pseudoknots and base triples,
Bioinformatics 14 (1998) 691--699.

\bibitem{Mathews:99}
D. Mathews, J. Sabina, M. Zuker, D.H. Turner, Expanded sequence dependence
of thermo-dynamic parameters improves prediction of RNA secondary structure,
J. Mol. Biol. 288 (1999) 911--940.

\bibitem{Waterman:78aa}
T.F. Smith, M.S. Waterman, RNA secondary structure,
Math. Biol. 42 (1978) 31--49.

\bibitem{Bon:08}
M. Bon, G. Vernizzi, H. Orland, A. Zee, Topological classification of RNA structures,
J. Mol. Biol. 379 (2008) 900--911.

\bibitem{rnag3}
J.E. Andersen, R.C. Penner, C.M. Reidys, M.S. Waterman, Topological classification and
enumeration of RNA structures by genus, Submitted (2011).

\bibitem{gfold}
C.M. Reidys, F.W.D. Huang, J.E. Andersen, R.C. Penner, P.F.Stadler, M.E. Nebel, Topology
and prediction of RNA pseudoknots,
Bioinformatics 27 (2011) 1076--1085.

\bibitem{Orland:02}
H. Orland, A. Zee, RNA folding and largenmatrix theory,
Nuclear Physics B 620 (2002) 456--476.

\bibitem{Zagier:95} D. Zagier, On the distribution of the number of cycles of elements in
symmetric groups, Nieuw Arch. Wisk. IV, 13 (1995) 489--495.

\bibitem{Loebl:08}
M. Loebl, I. Moffatt, The chromatic polynomial of fatgraphs and its categorification,
Adv. Math. 217 (2008) 1558--1587.

\bibitem{Penner:10}
R.C. Penner, M. Knudsen, C. Wiuf, J.E. Andersen, Fatgraph models of proteins,
Comm. Pure Appl. Math. 63 (2010) 1249--1297.

\bibitem{Massey:69}
W.S. Massey, Algebraic Topology: An Introduction, Springer-Veriag, New York, 1967.

\bibitem{Waterman:93}
R.C. Penner, M.S. Waterman, Spaces of RNA secondary structures,
 Adv. Math. 101 (1993) 31--49.

\bibitem{Penner:03}
R.C. Penner, Cell Decomposition and Compactification of Riemann's Moduli Space
in Decorated Teichm\"{u}ller theory, in Woods Hole Mathematics-perspectives in math and physics,
ed. by N. Tongring, R.C. Penner, World Scientific, Singapore, 2004, pp. 263--301 arXiv:
math.GT/0306190.

\bibitem{fenix2bb}
J.E. Andersen, F.W.D. Huang, R.C. Penner, and C.M. Reidys. Topology
of RNA-RNA interaction structures, Submitted (2012).

\bibitem{Bender:73}
E.A. Bender, Central and local limit theorems
applied to asymptotic enumeration,
J. Combin. Theory A. 15 (1973) 91--111.

\bibitem{Flajolet:07a}
P. Flajolet, R. Sedgewick, Analytic
combinatorics, Cambridge University Press, New York, 2009.

\bibitem{Han:11a} H.S.W. Han, C.M. Reidys, Combinatorics of $\gamma$-structures, Submitted (2011).

\bibitem{qinjing} J. Qin, C.M. Reidys, On topological RNA interaction structures, Submitted (2012).

\bibitem{Hille:62} E. Hille, Analytic Function Theory, Chelsea Publishing Company, Volume II, 1962.

\bibitem{Wall:04} C.T.C. Wall, Singular Points of Plane Curves, Cambridge University Press, 2004.

\end{thebibliography}

\end{document}